\documentclass{amsart}

\usepackage[letterpaper,margin=1in]{geometry}

\usepackage{amssymb,amsmath,amsfonts,graphicx,amsthm}

\usepackage{color}

\usepackage{subcaption}
\usepackage{amscd,verbatim}
\usepackage{enumerate}
\usepackage{float}
\usepackage{bm}

\newtheorem{theorem}{Theorem}%[section]

\newtheorem{lemma}[theorem]{Lemma}

\newtheorem{corollary}[theorem]{Corollary}

\newcommand{\R}[1]{{\rm I\! R}^{#1}}

\begin{document}

\title{Testing Gap $k$-planarity is NP-complete}
\author{John C. Urschel}
\address{Department of Mathematics, Massachusetts Institute of Technology, Cambridge, MA, USA.}
 \email{urschel@mit.edu}
 \author{Jake Wellens}
 \address{Department of Mathematics, Massachusetts Institute of Technology, Cambridge, MA, USA.}
 \email{jwellens@mit.edu}

\everymath{\displaystyle}

%%%%%%%%%%%%%%%%%%%%%%%%%%%%%%%%%%%%%%%%%%%%%%%%%%%%%%%%%%%%%%%5

%%%%%%%%%%%%%%%%%%%%%%%%%%%%%%%%%%%%%%%%%%%%%%%%%%%%%%%%%%%%%%%5 
\begin{abstract}
For all $k \geq 1$, we show that deciding whether a graph is $k$-planar is NP-complete, extending the well-known fact that deciding 1-planarity is NP-complete. Furthermore, we show that the gap version of this decision problem is NP-complete. In particular, given a graph with local crossing number either at most $k\ge 1$ or at least $2k$, we show that it is NP-complete to decide whether the local crossing number is at most $k$ or at least $2k$. This algorithmic lower bound proves the non-existence of a $(2-\epsilon)$-approximation algorithm for any fixed $k \ge 1$. In addition, we analyze the sometimes competing relationship between the local crossing number (maximum number of crossings per edge) and crossing number (total number of crossings) of a drawing. We present results regarding the non-existence of drawings that simultaneously approximately minimize both the local crossing number and crossing number of a graph.
\end{abstract}

\maketitle

\section{Introduction}

 Graph drawing is a central area of research in graph theory, with applications to network analysis, bioinformatics, circuit schematics, and software engineering, among other areas \cite{di1994algorithms}. For an introduction to the field, see  \cite{battista1998graph,kaufmann2003drawing}. However, in many applications, the corresponding graph is non-planar (e.g., small world networks). Even though the majority of these graphs cannot be drawn in the plane without edge crossings, minimizing the number of crossings that occur (either per edge or in total) in a planar embedding is an important task, for both practical and aesthetic reasons in application \cite{purchase2000effective,purchase2002empirical,ware2002cognitive}. This has led to a great deal of recent interest in graph drawings that minimize the number of crossings in some sense, and has made graph crossings one of the major areas of graph drawing research (see \cite{kobourov2017annotated} for a summary of the extensive literature regarding 1-planar graphs, \cite{didimo2019survey} for the study of drawings in terms of forbidden crossing configurations, and \cite{schaefer2018crossing} for a treatment of the overall field in general).

There are a number of competing measures of what exactly constitutes a drawing with few crossings. The total number of pairwise edge crossings of a drawing $D$ is called the crossing number of $D$, denoted by $cr(D)$. The minimum crossing number over all drawings $D$ of a graph $G$ is called the crossing number of $G$, denoted by $cr(G)$. While this is an important parameter, in some cases the maximum number of crossings per edge matters more than the total number of edge crossings. A drawing $D$ is $k$-planar if each edge participates in at most $k$ crossings, and a graph is $k$-planar if there exists a $k$-planar drawing of it. The maximum number of crossings per edge of a drawing $D$ is called the local crossing number of $D$, denoted by $lcr(D)$. The minimum local crossing number over all drawings $D$ of a graph $G$ is called the local crossing number of $G$, denoted by $lcr(G)$.

Both the crossing number and local crossing number are active areas of research with many open problems, such as the value of these quantities for specific graphs and the existence of approximation algorithms (for instance, see \cite{asplund2019k,pach2018note}). For a thorough treatment of both of these subjects, see \cite{schaefer2018crossing}. Despite the importance of the crossing number and local crossing number in producing a quality drawing of a graph, the majority of computational results regarding these two quantities have been algorithmic lower bounds rather than constructions. For instance, deciding if a graph has $cr(G)\le k$ for a fixed $k$ was shown to be NP-complete in \cite{garey1983crossing}. In \cite{grigoriev2007algorithms} and later independently in \cite{korzhik2013minimal}, it was shown that deciding whether a graph is $1$-planar is NP-complete (a third proof was later given in \cite{schaefer2018crossing}). Many generalizations of $1$-planarity, which we will not discuss in detail, including deciding $1$-gap-planarity \cite{bae2018gap} and deciding fan-planarity with fixed rotation system \cite{bekos2014recognition}, have also been shown to be NP-complete. Despite the similarity in terminology, the set of $k$-gap-planar graphs (graphs that can be drawn so that each crossing can be assigned to one of the two crossed edges in a way such that each edge is  assigned at most $k$ crossings) is unrelated to the gap $k$-planarity decision problem (given a graph $G$ with local crossing number either $\le k$ or $\ge ck$, for some fixed $c>1$, decide whether $lcr(G)\le k$) defined below. We note that there are a number of linear time planarity testing algorithms, most notably the edge addition method (see \cite{boyer2004cutting,de2006tremaux}).  The interplay between the local crossing number and crossing number of a drawing is an understudied area. Most notably, in \cite{chimani2019crossing}, the authors show the existence of $1$-planar (resp. quasi-planar, fan planar) order $n$ graphs $G$ for which $cr(G) = O(1)$ and every $1$-planar (resp. quasi-planar, fan planar) drawing $D$ of $G$ satisfies $cr(D) = \Omega(n)$. This result quantifies the ability to approximately minimize the crossing number over all drawings that minimize local crossing number. However, there is a lack of results regarding the existence or non-existence of drawings $D$ of a graph $G$ that approximately minimize both crossing number and local crossing number in some more general sense (for example, minimizing some function of $cr(D)/cr(G)$ and $lcr(D)/lcr(G)$).

\subsection{Our Contributions}
In this paper, we show that testing $k$-planarity is NP-complete for all $k \geq 1$. Furthermore, we show that the gap decision problem\\

\indent \indent  \noindent{\bf GAP $k$-PLANARITY} \\
\indent \indent  {\bf Input:} A graph $G = (V, E)$ with $lcr(G) \le k$ or $lcr(G) \ge 2k$. \\
\indent \indent  {\bf Output:} TRUE if $lcr(G) \le k$; FALSE otherwise. \\

\noindent is NP-complete. Our proof proceeds in two simple steps: first we prove that the multigraph variant of gap $k$-planarity is hard, and then we reduce to the graph case via a technique called edge subdivision. In \cite{korzhik2013minimal} the authors briefly sketch a possible way to modify their $1$-planarity proof to give hardness of $k$-planarity testing in multigraphs, but this approach is quite involved, and filling in the details appears somewhat difficult. Here, we provide a proof of an even stronger result, namely, hardness of the gap version of multigraph k-planarity testing, \\

\indent \indent  \noindent{\bf MULTIGRAPH GAP $k$-PLANARITY} \\
\indent \indent {\bf Input:} A multigraph $G = (V, E, \omega)$ with $lcr(G) \le k$ or $lcr(G) \ge 2k$. \\
\indent \indent  {\bf Output:} TRUE if $lcr(G) \le k$; FALSE otherwise. \\

 \noindent using a reduction from $3$-partition. Our gadget is inspired by a simplification of the technique used in \cite{grigoriev2007algorithms} to prove hardness of testing $1$-planarity. By proving our main result, we also provide an alternate proof of the hardness of $1$-planarity, one that does not rely on the machinery of $K_6$ that was needed in \cite{grigoriev2007algorithms}. It is fairly straightforward to verify that all of the decision problems considered in this paper are in NP, and therefore, it suffices to show hardness.
 
 In addition, we analyze the sometimes opposing optimization problems of minimizing the crossing number versus minimizing the local crossing number. We quantify the inability to simultaneously approximately minimize the crossing number and local crossing number of a drawing by constructing an infinite class of graphs $G = (V,E)$ for which $\frac{cr(D) lcr(D) }{cr(G) lcr(G)} \gtrsim |V|^{1/2}$ for all drawings $D$ of $G$. As a corollary of the crossing lemma, we also show that, for every graph $G$, there exists a drawing $D$ such that $\frac{cr(D) lcr(D) }{cr(G) lcr(G)} \lesssim |V|$.

\section{Definitions}
Here we provide formal definitions for some of the previously introduced terminology. A \emph{graph} $G=(V,E)$ is a pair $(V,E)$, where $V$ is a finite set and $E \subset \{e \subset V \, | \, |e| =2 \}$. A \emph{multigraph} $G = (V,E,\omega)$ is a graph $(V,E)$, paired with a function $\omega:E \rightarrow \mathbb{N}^+$ detailing the number of copies $\omega(e)$ of a given edge $e \in E$ which appear in $G$. A \emph{planar drawing} of a graph $G = (V,E)$ is an injective mapping from the vertex set $V$ to $\R{2}$ paired with a smooth plane curve $\phi_e:[0,1] \rightarrow \R{2}$ between $u$ and $v$ for all $e=(u,v) \in E$, with the condition that a vertex only intersects a curve at its endpoint. For the sake of brevity, we will refer to a planar drawing simply as a \emph{drawing} (or embedding). A \emph{crossing} between two edges $e$ and $e'$, $e \ne e'$, is a pair of values $t, t' \in (0,1)$ such that $\phi_e(t) = \phi_{e'}(t')$. A crossing between an edge $e$ and itself is a pair of values $t, t' \in (0,1)$, $t \ne t'$, such that $\phi_e(t) = \phi_e(t')$, and is referred to as a \emph{self-crossing}. A drawing is said to be \emph{normal} if each pair of edges has at most a finite number of crossings, no three edges cross at a single point, and there are no ``touching points" (a crossing which can be removed by a redrawing of the curves in any neighborhood of that point). A drawing is said to be \emph{good} if edges do not have self-crossings ($\phi(e)$ are simple curves), adjacent edges do not cross, and every pair of edges has at most one crossing. In this paper, we will assume all drawings are normal, as touching points can be removed and three edges crossing at a single point can always be perturbed so that the crossing number and local crossing number do not increase. However, we will not assume that drawings are good. While all drawings that minimize crossing number are known to be good, and, for all graphs $G$ with $lcr(G)\le 3$, there exists a good drawing which minimizes local crossing number, for $lcr(G)\ge 4$ there are examples for which all drawings that minimize local crossing number have crossings between adjacent edges \cite{pach2006improving}. We will assume that there are no self-crossings, as, similar to the touching points, these can be removed without increasing either the crossing number or local crossing number. All of the previous definitions for a graph are equally applicable for multigraphs, with the convention that two copies of an edge $e$ are considered two distinct edges, crossings between copies of an edge are allowed, and local crossing number counts the maximum number of crossings in any given copy of an edge. For a detailed discussion regarding the consequences of some of the above restrictions on drawings, as well as a number of related open problems, we refer the reader to \cite{schaefer2018crossing}.

\section{Testing Gap k-planarity is NP-complete}

In this section, we use multigraphs with edges of varying multiplicity to provide a simple proof that multigraph gap $k$-planarity is NP-hard, and then use a technique called edge subdivision to show that gap $k$-planarity is NP-hard, the main result of the paper.

Our proof is via a reduction from the $3$-partition problem. The $3$-partition decision problem asks whether a multiset $A$ of $3m$ integers can be partitioned into $m$ multisets $A_1,...,A_m$ such that the sum of the elements in each multiset is the same. Without loss of generality, one may assume that every integer is positive and strictly between $1/4$ and $1/2$ of the desired sum $B$. In addition, we assume that $B \ge 100$ and $m \ge 4$. One can easily verify that $3$-partition remains NP-hard under this mild restriction (for details, see \cite{garey2002computers}).

Our reduction converts an instance 
$$A = \{a_1,...,a_{3m}\}, \quad  \sum_{i=1}^{3m} a_i= Bm, \quad \frac{B}{4} < a_i < \frac{B}{2}, \; i = 1,...,3m,$$ 
of $3$-partition into a multigraph $G_A = (V,E, \omega)$ defined as follows: 
$$ V = \big\{t, c\big\} \bigcup \big\{ t_1,...,t_{3m}\big\} \bigcup \big\{c_1,...,c_{Bm}\big\} \bigcup \big\{s_1,...,s_{3m}\big\} \bigcup \big\{\ell_i^{j}\big\}_{i = 0,1,...,a_j}^{j=1,...,3m},$$
\begin{align*}
E &=  \big\{(t,t_{3i}) (t_{3i},c_{Bi}) (c_{Bi},c) \big\}_{i=1}^m \bigcup \big\{(t_i,t_{i+1})\big\}_{i=1}^{3m} \bigcup \big\{(c_i,c_{i+1})\big\}_{i=1}^{Bm} \\ &\qquad \bigcup \big\{(t,\ell_0^i),(\ell_0^i,s_i)\big\}_{i=1}^{3m} \bigcup \big\{(s_j,\ell_i^{j}), (\ell_i^{j},c)\big\}_{i = 1,...,a_j}^{j=1,...,3m}, 
\end{align*}
$$\omega(e) = \left\{  \begin{array}{l l} k & e \in   \big\{(t,\ell_0^i),(\ell_0^i,s_i)\big\}_{i=1}^{3m} \bigcup \big\{(s_j,\ell_i^{j}), (\ell_i^{j},c)\big\}_{i = 1,...,a_j}^{j=1,...,3m} \\
2k & e \in \big\{(t_i,t_{i+1})\big\}_{i=1}^{3m} \bigcup \big\{(c_i,c_{i+1})\big\}_{i=1}^{Bm} \\
5Bk & e \in  \big\{(t_,t_{3i}),(t_{3i},c_{Bi}) (c_{Bi},c) \big\}_{i=1}^m \end{array}, \right. $$
where $t_{3m+1}:=t_1$ and $c_{Bm+1}:=c_1$. We provide a visual example of this graph in Figure 1.

\begin{figure}\label{fig:ex}
	\center
	\includegraphics[width=4.5in]{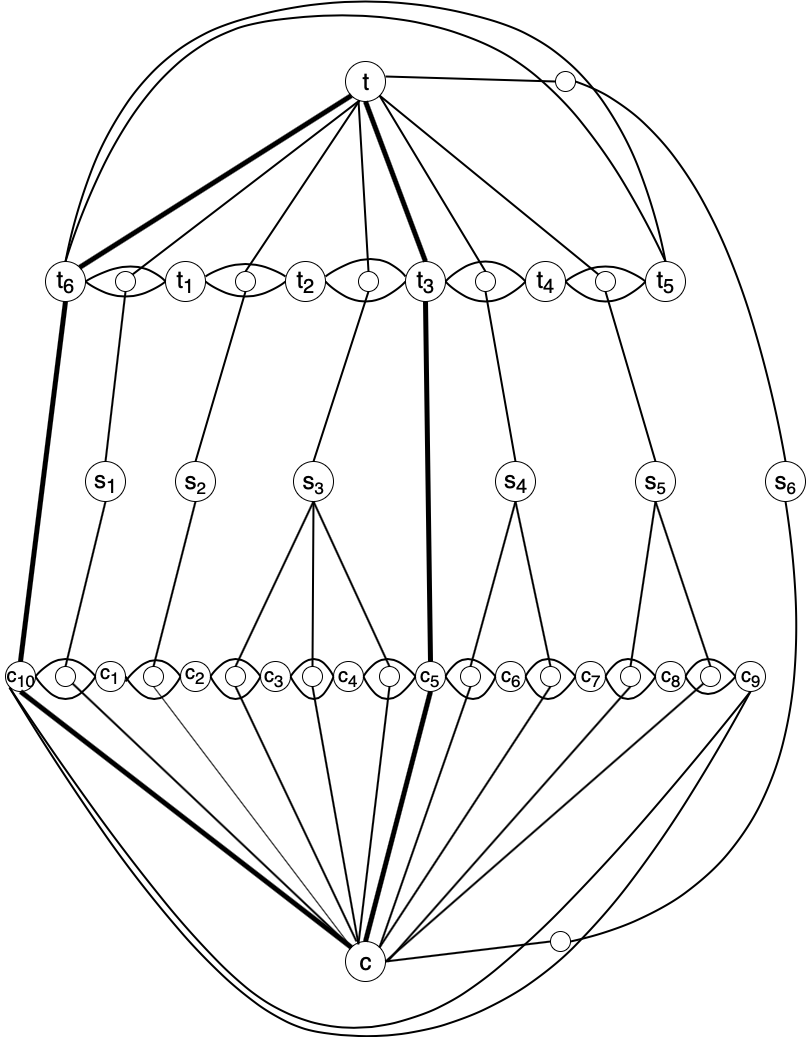} 
	\caption{A visual example of an embedding of $G_A$, where $A =\{a_1,a_2,a_3,a_4,a_5,a_6\}$, $a_1 = a_2 = a_6 = 1$, $a_4 = a_5 = 2$, $a_3 = 3$. The edges represent $k$ multi-edges and the bolded edges represent $5Bk$ multi-edges. The paths $P_1 = t \, t_3 \, c_5 \,c$ and $P_2 = t \, t_6 \, c_{10} \, c$ partition the sphere into two regions. The partition of the vertices $s_1,s_2,s_3,s_4,s_5,s_6$ into $s_1,s_2,s_3$ and $s_4,s_5,s_6$ by the paths $P_1$ and $P_2$ corresponds to the $3$-partition into multisets $\{a_1,a_2,a_3\}$ and $\{a_4,a_5,a_6\}$, each with sum equal to five.}
\end{figure}

While the formal definition is somewhat involved, the concept is rather straightforward. A graph drawing can be equivalently thought of as an embedding on the unit sphere $S^2 = \{ x \in \R{3} \, | \, \| x \| = 1\}$ (via stereographic projection). For simplicity, we will use this representation throughout the paper. Qualitatively, a drawing of $G_A$ can be thought of as a partition of the sphere, where $t$ is at the north pole $\vec{n} = (0,0,1)$, $c$ is at the south pole $\vec{s} = (0,0,-1)$, and the $m$ multi-edge paths of the form $P_i = t \, t_{3i} \, c_{Bi} \, c$, $i = 1,...,m$, partition $S^2$ into $m$ regions. Each star induced by the vertices $s_j, \ell^{j}_1,...,\ell^{j}_{a_j}$ corresponds to a number $a_j$ in $A$, $j=1,...,3m$. The multiset $A_i$ consists of the $a_j$ corresponding to the stars contained in the $i^{th}$ region of $S^2$. The cycle $C_t = t_1 \, t_2 \, ... \, t_{3m} $ paired with edges $(t,t_{3i})$ guarantee that each region has exactly three stars, and the cycle $C_c = c_1 \, c_2 \, ... \,c_{Bm}$ paired with edges $(c,c_{Bi})$ guarantee that in each region the number of leaves corresponding to the three stars (and therefore the sum of the corresponding $a_j$s) is exactly $B$. 

We are now prepared to provide a formal proof of our desired result.

\begin{theorem}\label{multi}
	Multigraph gap $k$-planarity ($k \ge 1$) is NP-complete.
\end{theorem}

We will prove two statements. First, we will show that if $A$ has a $3$-partition, then $G_A$ is $k$-planar. Second, we will show that if $A$ does not have a $3$-partition, then $lcr(G_A)\ge 2k$. These two results together complete the proof of Theorem \ref{multi}.

\begin{lemma} \label{lemma1}
	If $A$ has a $3$-partition, then $G_A$ is $k$-planar.
	\end{lemma}

\begin{proof}
Suppose that $A$ has a $3$-partition. We will explicitly describe a $k$-planar drawing $D$ of $G_A$. Place $t$ and $c$ at the north and south pole $\vec{n}$ and $\vec{s}$, respectively. Draw the $m$ multi-edge paths $P_i$ from $t$ to $c$ such that they do not cross. Draw the multi-edge cycles $C_t$ and $C_c$, again, in a non-crossing fashion. Each of the $m$ regions created by the paths $P_i$ corresponds to one of the $m$ multisets in the $3$-partition. In each of these regions, place the three star centers $s_i,s_j,s_k$ corresponding to the three elements of the corresponding multiset between $C_t$ and $C_c$. Place each of the vertices $\ell_0^i,\ell_0^j,\ell_0^k$, in the middle of the $2k$ copies of one of the three multi-edges of the path $t_{3\ell} \, t_{3 \ell +1} \, t_{3\ell + 2} \, t_{3\ell + 3}$. Because our partition of $A$ is a $3$-partition, there are a total of $B$ leaves connected to $s_i,s_j,s_k$. Place each leaf in the middle of the $2k$ copies of one of the $B$ multi-edges of the path $c_{B\ell} \, c_{B\ell+1} \, ... \, c_{B\ell + B}$. This is a $k$-planar drawing. An example of this layout is given in Figure 1.
\end{proof}

\begin{lemma} \label{lemma2}
	If $A$ does not have a $3$-partition, then $lcr(G_A)\ge 2k$.
\end{lemma}

\begin{proof}
Suppose that $A$ does not have a $3$-partition, but there exists a drawing $D$ of $G_A$ with $lcr(G_A)< 2k$. Without loss of generality, we may assume that $t$ and $c$ are at the north and south pole, respectively. Let $\delta>0$ be such that there is exactly one vertex and no edge crossings in $B(\vec{n}, \delta )$ and $B(\vec{s}, \delta ) $, where $B(x,r):=\{y \in S^2 \, | \, \|x - y\|<r\}$. There are $m$ multi-edge paths emanating from $t$ and reaching $c$, which are non-crossing in $B(\vec{n}, \delta )$ and $B(\vec{s}, \delta ) $. We will first look at the clockwise ordering of the $5Bk$ copies of the paths $P_i$, $i = 1,...,m$, in $B(\vec{n}, \delta )$. 

It may be the case that copies of $P_i$ interlace with copies of other paths in $B(\vec{n}, \delta )$. For any fixed $i$, the copies of $P_i$ partition $B(\vec{n}, \delta )$ into $5Bk$ regions, each containing some number of copies of other paths $P_j$, $j \ne i$ (which we will refer to as non-$P_i$ paths), the sum of which equals $5Bk(m-1)$. No two copies of $P_i$ can partition these non-$P_i$ paths into two regions containing at least $2(2k-1) +1$ non-$P_i$ paths each, as this would mean that the cycle created by the two associated copies of $(t,t_{3i})$ intersects with $2(2k-1)+1$ edge disjoint cycles consisting of the union of two paths in opposing regions. The number of non-$P_i$ paths ($5Bk(m-1)$) is more than three times $2(2k-1) +1$, so one of the original $5Bk$ regions must contain at least $2(2k-1)+1$ non-$P_i$ paths, and, therefore, all but at most $2(2k-1)$ non-$P_i$ paths. Let us denote this region by $R^*$.

 Each of the $5Bk-1$ regions which are at least $6(2k-1) +1$ regions away from $R^*$ cannot contain any non-$P_i$ paths, as such a non-$P_i$ path, combined with a non-$P_i$ path in $R^*$, constitutes a cycle of length six that intersects with at least $6(2k-1) +1$ edge-disjoint cycles created by two associated copies of $(t,t_{3i})$.
 
 By removing at most $24(2k-1)$ copies of each $P_i$, $i = 1,...,m$, we now have a drawing in which no copies of paths $P_i$, $P_j$ interlace in $B(\vec{n}, \delta )$ or $B(\vec{s}, \delta ) $, with at least $5Bk - 24(2k-1)>(5B-48)k$ copies of each $P_i$ remaining.

This creates a local clockwise ordering of the paths both at $B(\vec{n}, \delta )$ and $B(\vec{s}, \delta ) $. If these two orderings are not the same, then two multi-edge paths must cross, a contradiction, as each path is three multi-edges long, and $(5B-48)k > 3(2k-1)$. Next, we observe that the ordering in $B(\vec{n}, \delta ) $ and $B(\vec{s}, \delta ) $ must be the natural ordering (or the reversal of it), namely, the paths must be ordered $P_1,....,P_m$. Suppose to the contrary, that the ordering is such that there exists an $i$ such that $P_i$ and $P_{i+1}$ are not adjacent. Then there is a multi-edge cycle of length five consisting of $t \, t_{3i} \, t_{3i+1} \, t_{3i+2} \, t_{3i+3}$, which $(5B-48)k>5(2k-1)$ edge-disjoint copies of some multi-edge path $P_j$ must cross, a contradiction. See Figure 2 for a visual example.

\begin{figure}\label{fig:local}
	\center
	\includegraphics[width=4.5in]{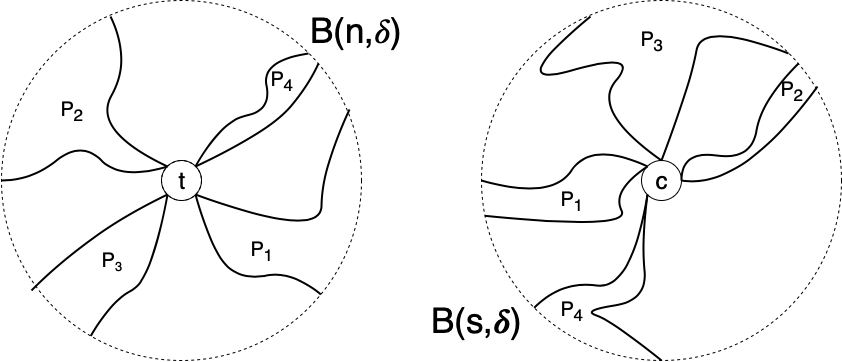} 
	\caption{A vertical view of $B(\vec{n},\delta)$ and $B(\vec{s},\delta)$. The clockwise ordering $P_1,P_3,P_2,P_4$ is the same for both $B(\vec{n},\delta)$ and $B(\vec{s},\delta)$. However, this cannot correspond to an embedding $D$ with $lcr(D)<2k$, as either $P_3$ or $P_4$ must cross the cycle $t \, t_{3} \, t_4 \, t_5 \, t_6$ to reach $c$.}
\end{figure}

Next, for every vertex $t_{3i}$ (and, similarly for $c_{Bi}$), we investigate the local structure of the at least $(5B-48)k$ copies of edges $(t,t_{3i})$ and $(t_{3i},c_{Bi})$. Let $\delta^*>0$ be such that the $\delta^*$ neighborhood of the location of $t_{3i}$ (denoted $N(t_{3i},\delta^*)$) only contains $t_{3i}$ and edges adjacent to $t_{3i}$, and no edge crossings. Using the same argument as above, we can remove some relatively small number of copies of $(t,t_{3i})$ and $(t_{3i},c_{Bi})$ and be left with edges which do not locally interlace.

The copies of $(t,t_{3i})$ partition $N(t_{3i},\delta^*)$ into $(5B-48)k$ regions, each of which contains some number of copies of $(t_{3i},c_{Bi})$. No two copies of $(t,t_{3i})$ can partition the copies of $(t_{3i},c_{Bi})$ into two regions containing at least $2(2k-1) +1$ copies of $(t_{3i},c_{Bi})$ each, as this would contradict $(2k-1)$-planarity. The number of copies of $(t_{3i},c_{Bi})$ ($\ge (5B-48)k$) is more than three times $2(2k-1) +1$, so one of the original $(5B-48)k$ regions must contain at least $2(2k-1)+1$ copies of $(t_{3i},c_{Bi})$, and, therefore, all but at most $2(2k-1)$ copies of $(t_{3i},c_{Bi})$. By removing at most $2(2k-1)$ copies of $(t,t_{3i})$ and $(t_{3i},c_{Bi})$ each, the local ordering in this neighborhood is such that copies of these two edges do not interlace. Therefore, by removing a total of at most $4(2k-1)$ copies of each path $P_i$, $i =1,...,k$, we now have edges which are locally well-ordered, and at least $(5B-48)k - 4(2k-1)> (5B-56)k$ copies of each $P_i$ remain.

We can now formally define each copy of the path, using the ordering of edges in $B(\vec{n},\delta)$ and labeling paths so that the ordering of copies of $(t,t_{3i})$ matches the ordering of $(t_{3i},c_{Bi})$ locally (and the same for $(t_{3i},c_{Bi})$ and $(c_{Bi},c)$ locally). See Figure 3 for an illustration.

\begin{figure}\label{fig:cross}
	\center
	\includegraphics[width=4.5in, height = .75 in]{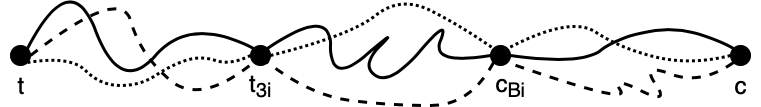} 
	\caption{Three copies of $P_1$, labeled so that the ordering agrees at each vertex.}
\end{figure}

We can now define a partition of $S^2$ based on the $m$ non-overlapping regions enclosed by the middle edge-disjoint copies of $P_i$ and $P_{i+1}$ (we define middle based on the ordering of $B(\vec{n},\delta)$). These middle edge-disjoint copies cannot cross any copy of any other path, as there are at least $(5B-56)k/2 -1 > 6(2k-1) -2$ edge-disjoint paths separating them. Let $R_i$ be the region defined by the middle copies of $P_i$ and $P_{i+1}$.

In addition, $R_i$ must fully contain a large number of copies of $P_i$ and $P_{i+1}$. In particular, the middle copy of $P_i$ cannot cross any copy of $P_i$ which is at least $6(2k-1)$ copies away in the initial ordering. Therefore, $R_i$ fully contains at least $(5B-56)k/2 -1 - 6(2k-1) >2Bk$ copies of $P_i$ and $P_{i+1}$ each. $R_i$ must fully contain the paths $t_{3i} \, t_{3i+1} \, t_{3i+2} \, t_{3i+3}$, otherwise there would be $2Bk >3(2k-1)$ edge-disjoint paths to cross. The same argument holds for $c_{Bi} \, c_{Bi+1} \, ... \, c_{Bi+B}$ by noting that $2Bk > B(2k-1)$.

We now consider the locations of the vertices $s_j$. Each vertex $s_j$ must be separated from $t$ by all multi-edge copies of $C_t$, otherwise there would be a $5$-cycle $t \, t_{3i} \, t_{3i+1} \, t_{3i+2} \, t_{3i+3}$ separating $a_j>Bk/4 > 5(2k-1)$ edge-disjoint paths from $s_j$ to $c$. Each region has at most three vertices $s_j$, otherwise one of the multi-edges in the path $t_{3i} \, t_{3i+1} \, t_{3i+2} \, t_{3i+3}$ would be crossed at least $2k$ times. Therefore each region has exactly three vertices $s_j$. 

The vertices $s_j$ are also separated from $c$ by all copies of $C_c$. Suppose this is not the case. If no copy of the multi-edge cycle $C_c$ separates $s_j$ from $c$, then there are $4k$ edge-disjoint multi-edge cycles separating $s_j$ from $t$, a contradiction. Then $s_j$ must be separated from $c$ by two copies of some edge $(c_i,c_{i+1})$, a contradiction to the $a_j>Bk/4 > 2(2k-1)$ edge-disjoint paths from $s_j$ to $c$. Therefore, $C_c$ separates $s_j$ from $c$.

Because $A$ does not have a $3$-partition, one of these regions, say $R_i$, must have three vertices $s_{j_1},s_{j_2},s_{j_3}$ such that the sum of their leaves exceeds $B$. However, there are only $B$ multi-edges in the path $c_{Bi} \, c_{Bi+1} \, ... \, c_{Bi+B}$, so one such multi-edge must have more than one path $s_j \, \ell^j_p \, c$ crossing it, a contradiction. The proof is complete.
\end{proof}

Although it is not necessary for the proof of Theorem \ref{multi}, one can also verify the stronger statement that $lcr(G_A) =k$ if and only if $A$ has a $3$-partition, otherwise $lcr(G_A)=2k$.

From here, we reduce from multigraph gap $2k$-planarity testing to gap $k$-planarity testing in a straightforward way. Given a multigraph $G = (V,E,\omega)$, we define the edge subdivision of $G$ to be the graph $G_* = (V_*,E_*)$ constructed by subdividing each edge of $G$ into two edges with a new vertex between them (i.e. replacing $e=(u,v) \in E$ by $(u,x_e),(x_e , v) \in E_*$, where $x_e \in V_*$ is a unique vertex for each copy of $e$). The key property of this edge subdivision is the following.

\begin{lemma} \label{lm:sub}
	Let $G_*$ be the edge subdivision of the multigraph $G$. Then 
	$$lcr(G_*) = \left\lceil \frac{lcr(G)}{2} \right\rceil.$$
\end{lemma}

\begin{proof}
	If $G_*$ is $k$-planar, then taking any $k$-planar drawing of $G_*$ and reversing the subdivision operation gives a drawing of $G$ which is clearly $2k$-planar. Conversely, given a $2k$-planar embedding of $G$, we can obtain from it a $k$-planar embedding of $G_*$ by placing the $x_e$ vertices ``in the middle" of the crossings on $e$ so that each segment has at most $k$ crossings. Therefore,
	$$ lcr(G) \le 2 lcr(G_*)  \quad \text{and} \quad    lcr(G_*) \le \bigg\lceil \frac{lcr(G)}{2} 
	\bigg\rceil,$$
 	which, by integrality of $lcr(G_*)$, implies that $lcr(G_*) = \bigg\lceil \frac{lcr(G)}{2}\bigg\rceil$.
\end{proof}

From here, our desired result immediately follows.

\begin{theorem}
	Gap $k$-planarity is NP-complete for all $k \ge 1$.
\end{theorem}

\begin{corollary}
	Deciding whether a graph is $k$-planar is NP-complete for all $k \ge 1$.
\end{corollary}

\section{Crossing number vs local crossing number}

In this section, we consider the problem of approximately minimizing both the crossing number and local crossing number of a drawing. In particular, we define
$$r(G) := \min_{D}  \; \frac{cr(D)}{cr(G)}  \;
 \frac{lcr(D)}{lcr(G)} ,$$
 where the minimum is taken over all drawings $D$ of a non-planar graph $G$. If $r(G)$ is small, then there exists a drawing of $G$ which simultaneously approximately minimizes both the total number of crossings and the maximum number of crossings per edge. However, if $r(G)$ is large, then these two minimization problems are clearly incompatible. We prove the following result.

\begin{theorem} \label{thm:lcrcr}
	Let $\mathcal{G}^{\neg p}_n$ be the set of non-planar graphs of order $n$. Then 
	$$c n^{1/2} \le \max_{G \in \mathcal{G}^{\neg p}_n} r(G) \le C n $$
	for all $n \ge 5$, for fixed constants $c,C$.
\end{theorem}

%This result is closely related to the results in \cite{chimani2019crossing}, namely, the existence of $1$-planar order $n$ graphs $G$ for which every $1$-planar drawing $D$ of $G$ satisfies $cr(D)/cr(G) \ge n/2-1$. They showed that these constructions are extremal in the sense that, for every graph with $lcr(G)=1$, there exists a $1$-planar drawing $D$ of $G$ with $cr(D)/cr(G)  \le n/2-1$. In contrast to \cite{chimani2019crossing}, our problem is more general, as we consider the problem of approximately minimizing both the crossing number and local crossing number of our drawing rather than minimizing crossing number over all drawings that minimize local crossing number, and we consider this quantity over all non-planar graphs rather than for $1$-planar graphs.

To prove the upper bound, we make use of the well-known crossing lemma.

\begin{theorem}[Crossing Lemma, \cite{ajtai1982crossing,leighton1984new}] \label{thm:crosslem}
	Let $G = (V,E)$ be a graph with $|E| \ge \lambda |V|$. Then
	$$cr(G) \ge (\lambda^{-2} - 3 \lambda^{-3}) \frac{|E|^3}{|V|^2}.$$
\end{theorem}

Let us temporarily restrict ourselves to graphs satisfying $|E| \ge \tfrac{9}{2} |V|$. By Theorem \ref{thm:crosslem}, 
$$ \frac{4}{243} \frac{ |E|^3}{|V|^2} \le cr(D) \le \frac{|E|}{2} lcr(D).$$
In addition, any drawing $D$ satisfying $cr(D) = cr(G)$ must also satisfy $lcr(D)< |E|$, and therefore $r(G) < |E|$. If this is not the case, then there are two edges which cross each other more than once. Removing these crossings locally decreases the crossing number, a contradiction. This produces a natural upper bound of
$$r(G) < \frac{|E|}{\tfrac{2}{|E|}cr(D)} = \frac{|E|^2}{2 cr(D)} \le \frac{|E|^2}{2 \Big( \tfrac{4}{243} \tfrac{ |E|^3}{|V|^2}\Big)} = \frac{243|V|^2}{8|E|}.$$
Combining our two bounds of $|E|$ and $\tfrac{243|V|^2}{8|E|}$, we obtain the upper estimate of Theorem \ref{thm:lcrcr}.

%
%\begin{figure}
%	\begin{subfigure}{0.32 \textwidth}
%		{\includegraphics[width = 1.5in]{5a}} 
%		\caption{}
%	\end{subfigure} 
%	\begin{subfigure}{0.32\textwidth}
%		{\includegraphics[width = 1.5in]{5c}}
%		\caption{}
%	\end{subfigure}
%	\begin{subfigure}{0.32\textwidth}
%		{\includegraphics[width = 1.5in]{5e}}
%		\caption{}
%	\end{subfigure} \\
%	\begin{subfigure}{0.32\textwidth}
%		{\includegraphics[width = 1.5in]{5b}}
%		\caption{}
%	\end{subfigure}
%	\begin{subfigure}{0.32\textwidth}
%		{\includegraphics[width = 1.5in]{5d}}
%		\caption{}
%	\end{subfigure}
%	\begin{subfigure}{0.32\textwidth}
%		{\includegraphics[width = 1.5in]{5f}}
%		\caption{}
%	\end{subfigure}
%	\caption{A visual demonstration of fixing an edge that crosses itself ((a)$\rightarrow$(d)), two edges that cross twice ((b)$\rightarrow$(e)), and three edges that cross at a single point ((c)$\rightarrow$(f)).}
%	\label{figpics}
%\end{figure}

To produce the lower bound, we give an infinite class of examples, based on adaptive edge subdivision of a multigraph version of $K_5$ (the smallest non-planar graph). In particular, let $G = (V,E)$, with
\begin{align*}
V &= \big\{u,v, w_1, w_2 , w_3 \big\} \bigcup \big\{x^i_{1,2},x^i_{2,3},x^i_{1,3}\big\}_{i=1}^{k^4} \bigcup \big\{x^{i,j}_{u,1},x^{i,j}_{u,2},x^{i,j}_{u,3}\big\}_{i=1,...,k^3}^{j = 1,...,k-1}\bigcup \big\{x^{i,j}_{v,1},x^{i,j}_{v,2},x^{i,j}_{v,3}\big\}_{i=1,...,k^3}^{j = 1,...,k-1}  , \\
E&=  \big\{ (u,x^{i,1}_{u,\ell}),(x^{i,1}_{u,\ell},x^{i,2}_{u,\ell}),...,(x^{i,k-1}_{u,\ell},w_\ell)  \big\}_{i = 1,...,k^3}^{\ell = 1,2,3} \bigcup \big\{ (v,x^{i,1}_{v,\ell}),(x^{i,1}_{v,\ell},x^{i,2}_{v,\ell}),...,(x^{i,k-1}_{v,\ell},w_\ell)  \big\}_{i = 1,...,k^3}^{\ell = 1,2,3} \\
&\qquad  \bigcup \big\{ (w_1,x^i_{1,2}), (x^i_{1,2}, w_2),(w_2,x^i_{2,3}), (x^i_{2,3}, w_3),(w_3,x^i_{1,3}), (x^i_{1,3}, w_1) \big\}_{i=1}^{k^4}  \bigcup \, (u,v), 
%E&= \{(u,w_i), (w_i,v)\}_{i=1}^3 \bigcup \{ (w_i,w_{i+1})\}_{i=1}^3 \bigcup \, (u,v) ,\\
%\omega(e) &= \left\{ \begin{array}{ll} 1 & e = (u,v) \\  k^3 & e \in \{(u,w_i), (w_i,v)\}_{i=1}^3  \\ k^4 & e \in  \{ (w_i,w_{i+1})\}_{i=1}^3   \end{array}   \right.,
\end{align*}
for some natural number $k$. This graph can be thought of as a multigraph of $K_5$ on the vertices $u,v,w_1,w_2,w_3$, where edges $(u,w_1),(u,w_2),(u,w_3)$, $(v,w_1),(v,w_2),(v,w_3)$ have multiplicity $k^3$; $(w_1,w_2),(w_2,w_3),(w_3,w_1)$ have multiplicity $k^4$; and $(u,v)$ has only one edge. Each copy of the edge $(w_i,w_j)$ is replaced by a standard edge subdivision $(w_i,x^{\ell}_{i,j}),(x^{\ell}_{i,j},w_{j})$, $\ell = 1,...,k^4$, but each copy of the edges of the form $(u,w_i)$ (and $(v,w_i)$, resp.) is replaced by a path of length $k$ given by $(u,x^{\ell,1}_{u,i}),(x^{\ell,1}_{u,i},x^{\ell,2}_{u,i}),...,(x^{\ell,k-1}_{u,i},w_i)$, $\ell = 1,...,k^3$. To refer to the path subdivision of an edge $e$ in $K_5$, we will simply write $P_e$.

We first consider two different drawings of $G$. Let $D_1$ be the drawing in which the cycles $C_w := P_{(w_1 ,w_2)} P_{(w_2 ,w_3)} P_{(w_1 ,w_3)} $ separates $u$ and $v$ and the only edge crossings are the single edge $(u,v)$ crossing all copies of $P_{(w_1,w_3)}$. In this case we have $lcr(D_1)= cr(D_1) = k^4$. Alternatively, let $D_2$ be the drawing in which $u$ and $v$ are on the same side of every copy of the cycle $C_w$, and the only edge crossings are all the copies of $P_{(u,w_2)}$ crossing all the copies of $P_{(v,w_1)}$. In this case, due to the subdivision into paths of length $k$, we have $lcr(D_2) = k^2$ and $cr(D_2) = k^6$. See Figure 4 for a visual representation of these two drawings.

\begin{figure} \label{fig:crlcr}
	\begin{subfigure}{0.39 \textwidth}
	{\includegraphics[width = 2 in]{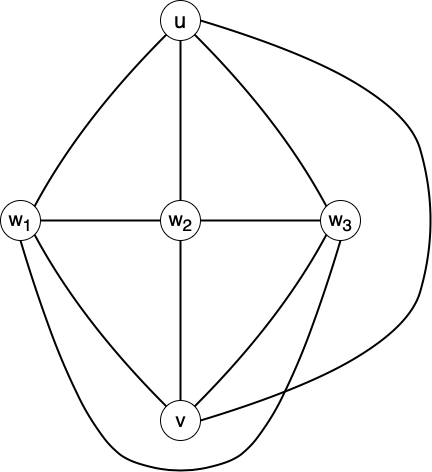}} 
	\caption{$D_1$}
	\end{subfigure} 	\begin{subfigure}{0.59\textwidth}
		{\includegraphics[width = 3 in]{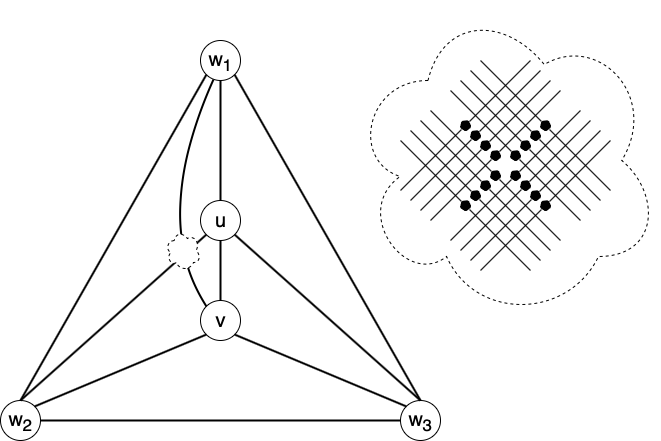}}
	\caption{$D_2$}
	\end{subfigure}
	\caption{A visual demonstration of the two drawings $D_1$ and $D_2$ of the graph $G$.}
	\label{examples}
\end{figure}

We will show that no drawing $D$ can produce a significantly better approximation to both crossing number and local crossing number than the two drawings described above. In particular, we will show that for all drawings $D$ of $G$, $lcr(D) \, cr(D) \ge \tilde c k^8$ for some $\tilde c$.

Suppose $D$ is such that at least $k^3$ edge disjoint copies of $C_w$ do not separate $u$ and $v$. If this is not the case, then $lcr(D) \, cr(D) \ge (k^4 - k^3)^2$.
Let $S_u$ be the ``star" created by $P_{(u,w_1)}$, $P_{(u,w_2)}$, and $P_{(u,w_3)}$ (with $S_v$ defined similarly). We have $k^3$ edge disjoint copies of $S_u$, $S_v$, and $C_w$, with $u$ and $v$ on the same side of every copy of $C_w$. Given one copy each of $S_u$, $S_v$, and $C_w$, with $u$ and $v$ on the same side of $C_w$, one of the three subgraphs must cross. This implies that there must be a total of at least $k^6$ crossings. As there are only $12k^4$ edges, the local crossing number is at least $k^2/12$.  Noting that $G$ has $6(k-1)k^3 + 3k^4 + 5$ vertices completes the proof of Theorem \ref{thm:lcrcr}.

\section{Concluding Remarks and Open Questions}

In this work, we have shown gap $k$-planarity testing is NP-complete for a gap of two, and quantified the ability to simultaneously approximately minimize the crossing and local crossing number of a drawing of a graph. However, a number of questions remain, especially in terms of approximation algorithms. For instance, \\

\begin{enumerate}[{Question} 1:]
\item What is the smallest $c_k \ge 2$ for which deciding whether a graph with local crossing number at most $k$ or at least $c_k k$ is NP-hard for a fixed $k\ge1$? \\

\item What is the smallest $c \ge 2$ for which there is a $c$-approximation algorithm for $lcr(G)$? \\

\item What is the asymptotic behavior of $\max_{G \in \mathcal{G}^{\neg p}_n} r(G)$ as a function of $n$?\\
\end{enumerate}
Answers to any of the above questions would mark a significant advancement in the understanding of graph drawings.

\section*{Acknowledgements}

The authors would like to thank Louisa Thomas for improving the style of the presentation, Erik Demaine for introducing us to hardness proofs of $1$-planarity, and Michel Goemans for interesting conversations on the subject. This research was initiated 
during an open problem session in MIT's Algorithmic Lower Bounds course. Research supported in part 
under ONR Research Contract N00014-17-1-2177.

{ \small 
	\bibliographystyle{plain}
	\bibliography{main}}

\end{document}